\documentclass[english,letterpaper]{article}
\usepackage[english]{babel}
\usepackage{amsmath,amsthm,amssymb}
\usepackage{latexsym}
\usepackage{enumerate}
\usepackage[latin1]{inputenc}
\usepackage{layout}
\usepackage{type1cm}
\usepackage[pdftex,plainpages=false,pdfpagelabels,pagebackref,colorlinks=true,citecolor=black,linkcolor=black,urlcolor=black,filecolor=black,bookmarksopen=true]{hyperref}

\newtheorem{teor}{Theorem}
\newtheorem{lema}{Lemma}
\newtheorem{prop}{Proposition}

\newtheorem{corol}{Corollary}
\theoremstyle{remark}
\newtheorem{remark}{Remark}

\title{Equilibrium states and zero temperature limit on topologically transitive countable Markov shifts}

\author{Ricardo Freire\thanks{Supported by FAPESP process 2011/16265-8.} \\
\footnotesize{Department of Mathematics, IME-USP, Brazil}\\
\footnotesize{\texttt{rfreire@usp.br}}
 \and Victor Vargas\thanks{Supported by CAPES. Parts of these results were in the author's PhD thesis.} \\
\footnotesize{Department of Mathematics, IME-USP, Brazil}\\
\footnotesize{\texttt{vavargascu@gmail.com}}
}

\begin{document}
\maketitle

\begin{abstract}
Consider a topologically transitive countable Markov shift and, let $f$ be a summable potential with bounded variation and finite Gurevic pressure. We prove that there exists an equilibrium state $\mu_{tf}$ for each $t > 1$ and that there exists accumulation points for the family $(\mu_{tf})_{t>1}$ as $t \to \infty$. We also prove that the Kolmogorov-Sinai entropy is continuous at $\infty$ with respect to the parameter $t$, that is $\lim_{t \to \infty} h(\mu_{tf})=h(\mu_{\infty})$, where $\mu_{\infty}$ is an accumulation point of the family $(\mu_{tf})_{t>1}$. These results do not depend on the existence of Gibbs measures and, therefore, they extend results of \cite{MaUr01} and \cite{Sar99} for the existence of equilibrium states without the BIP property, \cite{JMU05} for the existence of accumulation points in this case and, finally, we extend completely the result of \cite{Mor07} for the entropy zero temperature limit beyond the finitely primitive case.
\end{abstract}

\vspace*{5mm}


\vspace*{5mm}


\vspace*{5mm}

\section{Introduction}

The thermodynamic formalism is a branch of the ergodic theory that studies existence, uniqueness and properties of equilibrium states, that is, measures that maximize the value $h(\mu) + \int f d\mu$ where $h(\mu)$ is the Kolmogorov-Sinai entropy. If, for each $t > 1$, there is an unique equilibrium state $\mu_{tf}$ associated to the potential $tf$, an interesting problem is to study the accumulation points of the family $(\mu_{tf})_{t>1}$, as well as the behavior of the family as $t \to \infty$, since in statistical mechanics any accumulation points are the ground states of the system. 

We are interested in the case of topologically transitive countable Markov shifts. It is well known from \cite{BuSa03} that there is at most one equilibrium state, and existence is guaranteed usually by means of a strong condition on the dynamics. It is usually assumed that the incidence matrix satisfies properties like being finitely primitive, which is equivalent to the big images and preimages (BIP) property when the shift is topologically mixing, which is equivalent to the existence of Gibbs measures \cite{Sar03}. In general, it is very difficult to find a simple condition without such hypothesis. Our first results goes in this direction, and therefore extends the results for the existence of equilibrium states beyond the finitely primitive case, as the ones in \cite{MaUr01} and \cite{Sar99}.  

\begin{teor}\label{teor1} Let $\Sigma$ be a topologically transitive countable Markov shift and $f : \Sigma \to  \mathbb{R}$ a summable potential such that $V(f) < \infty$ and $P_G(f)<\infty$. Then, for any $t > 1$ there is an unique equilibrium state $\mu_{tf}$ associated to the potential ${tf}$. Also, we have that as $t\to\infty$, there exists accumulation points for the family $(\mu_{tf})_{t>1}$.
\end{teor}

The proof of this result is similar to one of the results in \cite{MaUr01}, that guarantees under similar conditions the existence of an eigenmeasure of the dual Ruelle operator. But it is not shown that in fact such eigenmeasure gives birth to an equilibrium state, as it is well known that is not always the case, in case of more interest see the example at the end of \cite{Sar03}.

In the case of Markov shifts with finite alphabet, the existence of accumulation points is trivial. For countable Markov shifts, Jenkinson, Mauldin and Urba\'nski in \cite{JMU05} gave conditions on the potential $f$ to guarantee the existence of accumulation points for the family $(\mu_{tf})_{t > 1}$, assuming that $\Sigma$ is a finitely primitive Markov shift with countable alphabet, and they also prove that these accumulation points are maximizing measures for $f$. Here, we do not focus on the maximizing property, since it is already known that if there exists an accumulation point for the equilibrium states as $t\to \infty$, then it is maximizing \cite{BMP15}. Again, we emphasize these results are in the context where there exists Gibbs measures.

We notice that existence of accumulation points does not imply the existence of the zero temperature limit of the equilibrium states, for which it is usually required another condition. One of the first works on the subject was done by Br\'emont in \cite{Bre03}, where it is proved the convergence when $f$ depends only on a finite number of coordinates, that is, $f$ is locally constant, and $\Sigma$ is a topologically transitive Markov shift with finite alphabet. More recently, Leplaideur in \cite{Lep05} gave an explicit form for the zero temperature limit when $\Sigma$ is a topologically Markov shift with finite alphabet\footnote{The proof uses an aperiodic incidence matrix, but it is essentially the same in the present context, as it can be seen also in the similar results in \cite{ChGU11}.}. Also, in \cite{ChH10}, following ideas in a different context in \cite{vER07}, it is given an example of a Lipschitz potential for which there is no convergence if we drop the requirement of locally constant. In fact, if we do not require the potential to be locally constant, the situation is much more complex, as it can be seen in  \cite{CRL14}, where given any two ergodic measures of same entropy, it is possible to construct a Lipschitz potential such that the equilibrium states family accumulates has both as accumulation points. In the case of Markov shifts with countable alphabet and the BIP property, the existence of the limit for locally constant potentials has been proved by Kempton \cite{Kem11}. The question whether the equilibrium states for a locally constant potential in the non BIP setting, that is, without the existence of Gibbs measures, is still open.

Also, under the same conditions of \cite{JMU05}, Morris has proved in \cite{Mor07} the existence of the limit $\lim_{t \to \infty} h(\mu_{tf})$ of the family of associated Kolmogorov-Sinai entropies and showed that this limit agree with the supremum of the entropies over the set of the maximizing measures of $f$. We are able to give an extension of this result in the same context of the previous theorem.

\begin{teor}\label{teor3} Let $\Sigma$ be a topologically transitive countable Markov shift and $f : \Sigma \to  \mathbb{R}$ a summable Markov potential such that $V(f) < \infty$ and $P_G(f)<\infty$. Then
\[h(\mu_{\infty}) = \limsup_{t \to \infty} h(\mu_{tf}) = \sup_{\mu \in \mathcal{M}_{max}(f)} h(\mu) \,,\]
where $\mu_{\infty}$ is an accumulation point of the family $(\mu_{tf})_{t>1}$. 
\end{teor}

As we have stated, the main development is that we do not depend on the existence of Gibbs measures for the whole space, since we use the condition that $f$ is a summable potential to guarantee the existence of the equilibrium state $\mu_{tf}$ for each $t > 1$. The best of our knowledge, this is the first proof of the convergence in the zero temperature limit beyond the finitely primitive case. 

Our proofs are based on an elaborate construction similar to a diagonal argument. We approximate our countable Markov shift by compact invariant subshifts $\Sigma_k$ and use the results in  \cite{BiFr14} to locate the ground states on a well determined compact subshift $\Sigma_{k_0}$. Then, as $k \to \infty$, we use the fact that the potential is summable and the variational principle, through a fine control of the entropy, to show that there exists equilibrium states for each $t>1$ and their accumulation points as $t\to\infty$. As usual, this is not as simple as it might seem at a first look, in particular since $\Sigma$ is not $\sigma$-compact, and also since in this general setting we lose some important tools in classical thermodynamic formalism, as the Gibbsianess of the equilibrium states and some strong properties on the Ruelle operator. We give some more detail on this on the next section.

The paper is organized as follows. In the next section, we give the basic definitions and notations that we use in the proofs and results, as well more details on the setting we are working on. In section 3 we construct the basic approximation by compacts that we have to deal with to prove our results. In section 4, we use the existence of the equilibrium states on the compact case to show the existence of an unique equilibrium state in the case of countable Markov shifts, therefore proving theorem 1. It is in section 4 that most of our main hypothesis show their strength, since we have to make careful estimates to control the entropy and prove the existence of an equilibrium state. In section 5, we use the existence of accumulation points in zero temperature to show the existence of the entropy zero temperature limit, therefore proving theorem 2.

\section{Preliminaries}

Let $\mathcal{A}$ be a countable alphabet and $\textbf{M}$ a matrix of zeros and ones indexed by $\mathcal{A} \times \mathcal{A}$. Let $\Sigma$ be a topological Markov shift on the alphabet $\mathcal{A}$ with incidence matrix $\textbf{M}$, that is
\[
\Sigma = \{x \in \mathcal{A}^{\mathbb{N} \cup \{0\}} : \textbf{M}_{x_i, x_{i+1}} = 1\} \,,
\]
 where the dynamics is given by the shift map  $\sigma : \Sigma \to \Sigma$, that is the map defined by $\sigma((x_n)_{n \geq 0}) = (x_n)_{n \geq 1}$. Just to ease the calculations we suppose that $\mathcal{A} = \mathbb{N}$. Recall that a word $\omega$ is admissible if $\omega$ appears in $x \in \Sigma$. It is well known that $\Sigma$ is a metric space with topology compatible to the product topology. Also, the topology has a sub-base made by cylinders that are both open and closed sets defined as
 \[
 [\omega] = \{x \in \Sigma: x \textrm{ starts with the word } \omega \}\,.
 \]

From now on, we assume that $\Sigma$ is topologically transitive in the sense that $\sigma$ is topologically transitive.  It is well known that this is equivalent to $\textbf{M}$ being irreducible, which is also equivalent to the fact that given any symbols $i, j \in \mathcal{A}$ there is an admissible word $\omega$ such that $i \omega j$ is also admissible.
 
 Fix a potential $f : \Sigma \to \mathbb{R}$, for each $n \geq 1$ we define the $n$-th variation of $f$ as
\[
V_n(f) = \sup\{|f(x) - f(y)| : x_0 \ldots x_{n-1} = y_0 \ldots y_{n-1}\} \,.
\]
We say that $f$ has summable (or bounded) variations if 
$$V(f) = \sum_{n \in \mathbb{N}} V_n(f) < \infty \,.$$

 We say that $f$ is coercive if $\lim_{i \to \infty} \sup f|_{[i]} = -\infty$, 
 and that $f$ is summable if satisfies 
\begin{equation}
\sum_{i \in \mathbb{N}} \exp(\sup (f|_{[i]})) < \infty \,.
\label{eqsum}
\end{equation}
Observe that if $f$ is summable then $f$ coercive. Moreover, it is shown in \cite{Mor07}\footnote{One can easily notice the BIP property is not required for this.} that the summability condition implies, for any $t > 1$, that
\begin{equation}
\sum_{i \in \mathbb{N}} \sup (-tf|_{[i]})\exp(\sup (tf|_{[i]})) < \infty \,.
\label{eqsumt}
\end{equation}
It follows from \cite{BiFr14} that the summability condition allow us to guarantee existence of maximizing measures for the potential $f$, since it is coercive. Also, it is the key ingredient to prove in this paper the  existence of the equilibrium state associated to $tf$ for each $t > 1$. 

From now on, we assume that $V(f) < \infty$ and that $f$ is summable. Observe that in this case $f$ is uniformly continuous and bounded above.

We use the following notation 
\[
\mu(f) := \int f d\mu\,.
\] 

Define
\begin{equation}
\beta = \sup \{\mu(f) : \mu \in \mathcal{M}_{\sigma}(\Sigma) \} \,,
\label{eqsup}
\end{equation}
where $\mathcal{M}_{\sigma}(\Sigma)$ is the set of $\sigma$-invariant borel  probability measures on $\Sigma$. When $\mu(f) = \beta$, we say that $\mu$ is a maximizing measure, and denote the set of maximizing measures by $\mathcal{M}_{\max}(f)$. For each $a \in \mathcal{A}$, let
\[
Z_n(f,a) = \sum_{\sigma^n x = x} \exp(S_n f(x)) 1_{[a]}(x) \,,
\]
with $S_n f(x) = \sum^{n-1}_{i = 0} f(\sigma^i(x))$. Then, the Gurevic pressure of $f$ is defined as
\[
P_G(f) = \lim_{n \to \infty} \frac{1}{n} \log Z_n(f,a) \,.
\]
Since $\Sigma$ is topologically transitive, the above pressure definition is independent from the choice of $a \in \mathcal{A}$,  furthermore $-\infty < P_G(f) \leq \infty$ and satisfies the variational principle (See \cite{Sar99} and \cite{FFM02})
\begin{equation}\label{VP}
P_G(f) = \sup \{ h(\mu) + \mu(f) : \mu \in \mathcal{M}_{\sigma}(\Sigma) \textrm{ and } \mu(f) >  - \infty \} \,.
\end{equation}
So, in this case we have that the Gurevic pressure is the same as the topological pressure, and then the hypothesis that require finite pressure can be read in any way. 

Also, $P_G(f)$ satisfies the another variational principle (see \cite{Sar99} and \cite{FFM02})
\begin{equation}\label{VPCompact}
P_G(f) = \sup \{ P_G(f|_{\Sigma^\prime}) : \Sigma^\prime \textrm { is a  compact subshift of } \Sigma \} \,.
\end{equation}

A measure $\mu \in \mathcal{M}_{\sigma}(\Sigma)$ is called an equilibrium state associated to $f$ when $h(\mu) + \mu(f)$ is well defined and reaches the supremum in \eqref{VP}, that is
\[
P_G(f) = h(\mu) + \mu(f) \,.
\]

We say that a measure $\mu \in \mathcal{M}_{\sigma}(\Sigma)$ is an invariant  Gibbs state associated to $f$ if there is a constant $C > 1$ such that for any $x \in \Sigma$ and each $n \geq 1$
\begin{equation} \label{gibbs}
C^{-1} \leq \frac{\mu[x_0 \ldots x_{n-1}]}{\exp(S_n f(x) - nP_G(f))} \leq C \,.
\end{equation}

When $\Sigma$ is compact the equilibrium states and the invariant Gibbs states are unique and agree, see, for example,  \cite{MaUr01}. Furthermore, when $\Sigma$ is compact and the potential $f$ has summable variations, we can choose 
\begin{equation} \label{C2}
C = \exp(4V(f)) \,. 
\end{equation}

Recall that the Ruelle operator $L_f$ associated to $f$ is defined as  
\[
(L_f g)(x) = \sum_{\sigma y = x} \exp(f(y))g(y) \,,
\]
and since $f$ is summable, $P_G(f) <\infty$ and $V(f) < \infty$, $L_f$ is well defined. Also, one can look into the dual operator $L^*_f$ defined as  
\[
(L^*_f \mu)(g) = \mu(L_f g) \,,
\]
with $g : \Sigma \to \mathbb{C}$. These operators and its properties are the base for the following well known results.

It is proved in \cite{BuSa03} the uniqueness of the equilibrium state, whenever they exist, assuming $\Sigma$ is a topologically transitive countable Markov shift and $f$ is bounded above with summable variation and $P_G(f) < \infty$. In this case, we will denote by $\mu_f$ the unique equilibrium state associated to $f$.

Furthermore, it is proved in \cite{BuSa03} that if the equilibrium state exists, then $$\mu_f = hd\nu$$ where $h$ is the eigenfunction of $L_f$ associated to the eigenvalue $\lambda = e^{P_G(f)}$ and $\nu$ is the eigenmeasure of $L^*_f$ associated to the same eigenvalue. In \cite{MaUr01}, it is proved that assuming $f$ is a summable potential, then there exists an eigenmeasure $\nu$ for $L^*_f$, but it is not proved that it will, in fact, result in an equilibrium state. 

In this paper, we take a different approach. Instead of finding an eigenmeasure to $L^*_f$ and working out its properties to prove it is in fact an equilibrium state under our hypothesis, we show that the equilibrium states on a suitable sequence of compact Markov shifts accumulate on a measure that is, in fact, an equilibrium state on the countable Markov shift, by the argument we have sketched in the previous section. This approach gives us some minor benefits, in particular the fact that we can easily realize that the equilibrium states family is tight.

We can suppose, w.l.o.g. that $f \leq 0$, since $f$ is bounded above by the fact that it is  coercive and $V(f)< \infty$, so we can consider $ f - \sup f$ instead of $f$.

\section{Compact subshifts approximation}

We suppose that $f$ is a summable potential such that $P_G(f) < \infty$ and $V(f) < \infty$ from now on. Our aim is to prove the existence of equilibrium states for $tf$ and $t>1$ and accumulation points for such sequence of equilibrium states as we approach the zero temperature limit on topologically transitive countable Markov shifts. We accomplish this using an approximation by compact subshifts of $\Sigma$ and its Gibbs equilibrium states in each compact subspace.

We can choose a sequence $(\Sigma_k)_{k \in \mathbb{N}}$ of compact topologically transitive subshifts of $\Sigma$ such that for any $k \in \mathbb{N}$ we have $\Sigma_k \subsetneq \Sigma$ and $\Sigma_k \subsetneq \Sigma_{k+1}$ and such that the variational principle \eqref{VPCompact} can be resumed to
\begin{equation} \label{VPSigmas}
P_G(f) = \sup \{ P_G(f|_{\Sigma_k}) : k \in \mathbb{N} \} \,.
\end{equation} 

In fact, for each $k \in \mathbb{N}$ we can choose $\mathcal{A}_k := \{0, \ldots, m_k\} \cup \{\text{finite elements}\}$, where $m_k$ is a strictly increasing sequence in $\mathbb{N}$ and the choice of finite elements, which depends on $k$, is made to allow to connect any of the symbols in $\{0, \ldots, m_k\}$. It is always possible to choose such finite alphabet since $\Sigma$ is topologically transitive. We can also choose $m_k = \max\{j : j \in \mathcal{A}_{k-1}\} + 1$ for $k \geq 1$, which assures that $\mathcal{A}_k \subset \mathcal{A}_{k+1}$ and $\bigcup_{k\in\mathbb{N}}\mathcal{A}_k = \mathcal{A}$. This construction is classical and appears, for example, in \cite{MaUr03}.

It is not difficult to show that for all $k$ we have $\Sigma_k \subset \Sigma$, $\sigma|_{\Sigma_k}$ is also topologically transitive and that this construction satisfies \eqref{VPSigmas}, since any compact subshift of $\Sigma$ is contained in some $\Sigma_k$ for large $k$.

Also, notice that our proof bellow works mainly because \eqref{VPSigmas} is true in this sequence, and it does not require that $\Sigma$ is in fact decomposed into this sequence, which would be impossible, since $\Sigma$ is not $\sigma$-compact.

Let $f_k := f|_{\Sigma_k} : \Sigma_k \to \mathbb{R}$ the restriction of $f$ to $\Sigma_k$ and 
\[
\beta_k := \sup \{\mu(f_k) : \mu \in \mathcal{M}_{\sigma}(\Sigma_k) \} \,.
\]
Since $f$ is summable, we have that $f$ is coercive. Therefore, the main theorem in \cite{BiFr14} says that there is a finite set $F \subset \mathcal{A}$ such that \eqref{eqsup} becomes 
\[
\beta = \sup \{ \mu(f) : \mu \in \mathcal{M}_{\sigma}(\Sigma_F)\} \,,
\]
where $\Sigma_F$ is the restriction of $\Sigma$ to the alphabet $F$, and since $\Sigma_F$ is compact then $\mathcal{M}_{max}(f) \neq \emptyset$. Moreover, it also implies that for any $\mu \in \mathcal{M}_{max}(f)$ we have $supp(\mu) \subset \Sigma_F$. 
 
Denote by $P : [1, \infty)  \to \mathbb{R}$ the function $t \mapsto P(t) = P_G(tf) < \infty$, and consider the sequence $(P_k)_{k \in \mathbb{N}}$ such that $P_k : [1, \infty)  \to \mathbb{R}$ is the function $t \mapsto P_k(t) = P_G(tf_k)$. Since for each $k \in \mathbb{N}$ 
\[
\sum_{\sigma^n x = x}\exp(S_n tf(x)) 1_{[a]}(x) \geq \sum_{\sigma^n x = x}\exp(S_n tf(x)) 1_{[a]}(x) 1_{\Sigma_k}(x) \,,
\]
then for any $t \geq 1$ we have $P_k(t) \leq P_{k+1}(t) \leq P(t)$, and it follows from \eqref{VPSigmas} that
\begin{equation} \label{VPApproximation}
P(t) = \sup \{P_k(t) : k \in \mathbb{N} \} \,.
\end{equation}

 Bellow we show that the sequence $(\mu_{tf_k})_{k \in \mathbb{N}}$ in $\mathcal{M}_{\sigma}(\Sigma)$ has a convergent subsequence. For this, we  prove that this sequence is tight. Let us recall that a subset $\mathcal{K} \subset \mathcal{M}_{\sigma}(\Sigma)$ is tight if for every $\epsilon > 0$ there is a compact set $K \subset \Sigma$ such that $\mu(K^c) < \epsilon$ for any $\mu \in \mathcal{K}$

\begin{lema}\label{lematight} For each $t > 1$ the equilibrium states sequence $(\mu_{tf_k})_{k \in \mathbb{N}}$ is tight.
\end{lema}

\begin{proof} Our proof is similar to the proof in \cite{JMU05}. Let $\epsilon > 0$ and
\[
K = \{x \in \Sigma : 1 \leq x_m \leq n_{m} \text{ for each $m \in \mathbb{N}$}\} \,,
\]
where $(n_m)_{m \in \mathbb{N}}$ in $\mathbb{N}$ is an increasing sequence.
Then the set $K$ is compact in the product topology and satisfies
\begin{align}
\mu_{tf_k}(K^c)
                       &= \mu_{tf_k}(\bigcup_{m \in \mathbb{N}} \{x \in \Sigma : x_m > n_m\})\nonumber \\
                       &\leq \sum_{m \in \mathbb{N}} \sum_{i > n_m} \mu_{tf_k}(\{x \in \Sigma : x_m = i\})\label{eqcom1} \\
                       &= \sum_{m \in \mathbb{N}} \sum_{i > n_m} \mu_{tf_k}([i]) \,.\nonumber 
\end{align}
We choose $(n_m)_{m \in \mathbb{N}}$ such that	
\[
\sum_{i > n_m} \mu_{tf_k}([i]) < \frac{\epsilon}{2^{m + 1}} \,.
\]
Let $\mu \in \mathcal{M}_{\sigma}(\Sigma)$ such that $S = \mu(f)$ satisfies $-\infty < S < \infty$. It is sufficient to choose $\mu = \frac{1}{p}\sum^{p - 1}_{j = 0} \delta_{\sigma^j \bar{x}}$ with $\bar{x} \in Per_p(\Sigma_0)$. Notice that since $\Sigma_0 \subset \Sigma_k$ for all $k \in \mathbb{N}$, we can consider $\mu$ to be a measure well defined both in $\Sigma$ or in $\Sigma_k$ for any $k$. Let $S_k = \mu(f_k)$ and by the previous comment, we have that $S_k$ is well defined, and it is also clear that $S_k = S$ for any $k \in \mathbb{N}$, since $\mu(f)$ is the ergodic average of $f$ over a periodic orbit in $\Sigma_0$ and $f$ and $f_k$ are the same in $\Sigma_0$ for any $k$. Furthermore, we have that
\begin{equation}
P_k(t) - tS_k = P_G(t(f_k - S_k)) \geq h(\mu) + t(\mu(f_k) - S_k) = h(\mu) \geq 0 \,.
\label{eqcom2}
\end{equation}

Since each $\Sigma_k$ is compact, from \eqref{gibbs}, \eqref{C2} and the fact that $\exp(4V(tf_k)) \leq \exp(4tV(f)) < \infty$, then for all $x \in [i]$ we have that
\begin{equation}\label{gibbs_in_the_1_cylinders}
\exp(-4V(tf)) \leq \frac{\mu_{tf_k}[i]}{\exp(tf_k(x) - P_k(t))} \leq \exp(4V(tf)) \,.
\end{equation}

Recall that $S_k=S$ for all $k$ and by \eqref{eqcom2}, we obtain that for each $x \in [i]$
\begin{align*}
\mu_{tf_k}([i]) &\leq \exp(4tV(f) + tf_k(x) - P_k(t)) \\
								&\leq \exp(4tV(f) + t\sup f |_{[i]} - P_k(t)) \\
                &= \exp(t(4V(f) + \sup f |_{[i]} - S_k)) \exp(tS_k- P_k(t)) \\
                &\leq \exp(t(4V(f) + \sup f |_{[i]} - S)) \,. \\                
\end{align*}
Since the potential $f$ is coercive, then for $i$ large enough we have the inequality $4V(f) + \sup f |_{[i]} - S \leq 0$, and since $t > 1$ we have to
\begin{equation} \label{ut_is_tight}
\mu_{tf_k}([i]) \leq \exp(4V(f) + \sup f |_{[i]} - S) \,.
\end{equation}

Moreover, from the summability condition \eqref{eqsum} we can suppose that the sequence $(n_m)_{m \in \mathbb{N}}$ satisfies
\[
\sum_{i > n_m} \exp(\sup f |_{[i]}) < \frac{\epsilon}{2^{m + 1}} \exp(S - 4V(f)) \,.
\]
Therefore, for any $k$
\begin{equation}
\sum_{i > n_m} \mu_{tf_k}([i]) \leq \sum_{i > n_m} \exp(4V(f) + \sup f |_{[i]} - S) < \frac{\epsilon}{2^{m + 1}} \,.
\label{eqcom3}
\end{equation}
 From \eqref{eqcom1} and \eqref{eqcom3}, we conclude that					
\[
\mu_{tf_k}(K^c) < \sum_{m \in \mathbb{N}} \frac{\epsilon}{2^{m + 1}} = \epsilon \,.
\]
\end{proof}

\begin{remark}\label{rem1}
The last passage in the proof of lemma \ref{lematight} is the key point where the summability condition is in fact needed. It is not the only point where the summability condition is essential, but it would be interesting to know if it is possible to use a different argument assuming only that $f$ is coercive, even if we need some additional regularity on $f$.
\end{remark}

\section{Proof of theorem 1}

In this section we prove our first theorem, which is a consequence of the summability condition \eqref{eqsum}, the finiteness of the pressure and the tightness of the sequence $(\mu_{tf_k})_{k \in \mathbb{N}}$. 

By the \emph{Prohorov's theorem}, there exists a subsequence $(\mu_{tf_{k_m}})_{m \in \mathbb{N}}$ of the sequence $(\mu_{tf_k})_{k \in \mathbb{N}}$ and a measure $\mu_t \in \mathcal{M}_{\sigma}(\Sigma)$ such that 
\[
\mu_t = \lim_{m \to \infty} \mu_{tf_{k_m}}\,.\]

Our aim is  prove that $\mu_t$ is, indeed, an equilibrium state associated to the potential $tf$. For this purpose, we have to estimate both $\mu_t(tf)$ and $h(\mu_t)$ to prove they are finite, so that $h(\mu_t) + \mu_t(tf)$ is well defined and then we can finally verify that they are also maximal. We begin with $\mu_t(tf)$.

From \eqref{gibbs_in_the_1_cylinders}, for each $m \in \mathbb{N}$ we have 
\[
\mu_{tf_{k_m}}[i] \leq \exp(4tV(f) + \sup(tf|_{[i]}) - P_{0}(t)) \,,
\]
and notice that for each $i \geq 1$ we have $\partial [i] = \emptyset$, where $\partial [i]$ the topological boundary of the cylinder $[i]$. Therefore, the cylinder $[i]$ is a continuity set of $\mu_t$, and taking the limit as $m \to \infty$ we obtain 
\[
\mu_t[i] \leq \exp(4tV(f) + \sup(tf|_{[i]}) - P_{0}(t)) \,.
\]

Then, for any $t > 1$, we have
\begin{align}
\mu_t(-tf) &= \mu_t \left(\sum_{i \in \mathbb{N}} -tf|_{[i]} \right)\nonumber \\
					&\leq \sum_{i \in \mathbb{N}} \sup(-tf|_{[i]}) \mu_t[i] \label{eq1thm1}\\
					&\leq C_t \sum_{i \in \mathbb{N}} \sup(-tf|_{[i]}) \exp(\sup(tf|_{[i]})) < \infty \,. \nonumber
\end{align}

So $\mu_t(tf)$ is finite for any $t>1$ and in fact we can approximate it well through our compact sequence, as the following lemma shows. 

\begin{lema} \label{lema_lim_ut}For each $t > 1$ we have that $\mu_t(tf) = \lim_{m \to \infty} \mu_{tf_{k_m}}(tf)$.
\end{lema}
\begin{proof} Observe that this result is not a direct consequence of the weak* convergence of the sequence $(\mu_{tf_{k_m}})_{m \in \mathbb{N}}$ since the potential $tf$ is not bounded bellow.

For each $N \in \mathbb{N}$, denote by $f^{N}:\Sigma \to \mathbb{R}$ the function
\[
f^N(x) :=
\begin{cases}
f(x) & \textrm{ if } f(x) \ge -N \\
-N & \textrm{ otherwise}
\end{cases}
\]
and set $g^{N}:\Sigma \to \mathbb{R}$ such that $f = f^N + g^N$.

Notice that both $f^N$ and $g^N$ are continuous, and that $f^N$ is bounded. Also, it is easy to verify that $|g^N(x)| \le |f(x)|$ and $g^N(x) \le 0$ for all $x \in \Sigma$. 

First, let us show that given $\epsilon > 0$, there is $N_0 \in \mathbb{N}$ such that $|\mu_{tf_{k_m}}(t g^N)| < \frac{\epsilon}{4}$ and $|\mu_{t}(t g^N)| < \frac{\epsilon}{4}$ for all $N \geq N_0$ and any $m$. 

We prove it for $\mu_t$ since for $\mu_{t f_{k_m}}$ it is similar and, in fact, it works with the same $N_0$. 

Given $n_0 \in \mathbb{N}$, since $f$ is coercive and $V(f)<\infty$ we can choose $N_0=N_0(n_0)$ large enough such that $g^N\vert_{[i]}=0$ if $i \le n_0$, for all $N \ge N_0$. In this way, we can proceed as in \eqref{eq1thm1} and we get 
\begin{align*}
|\mu_{t}(tg^N)| = \mu_t(-tg^N) &= \mu_t \left(\sum_{i > n_0} -tg^N|_{[i]} \right) \\
					&\leq \sum_{i > n_0} \sup(-tg^N|_{[i]}) \mu_t[i] \\
					&\leq C_t \sum_{i > n_0} \sup(-tg^N|_{[i]}) \exp(\sup(tf|_{[i]})) \\
 & \leq C_t \sum_{i > n_0} \sup(-tf|_{[i]}) \exp(\sup(tf|_{[i]}))\,,
\end{align*}
and if we choose $n_0$ big enough, we get $|\mu_{t}(tg^N)| < \frac{\epsilon}{4}$, for $N \geq N_0$ and all $m$ as desired.

Finally, to conclude the proof, given $\epsilon > 0$, fix $N \geq N_0$ as above and we have that
\begin{align*}
|\mu_{tf_{k_m}}(tf) - \mu_{t}(tf) | & \le |\mu_{tf_{k_m}}(t f^N) - \mu_{t}(t f^N) | + |\mu_{tf_{k_m}}(t g^N) - \mu_{t}(t g^N) | \\
&\le |\mu_{tf_{k_m}}(t f^N) - \mu_{t}(t f^N) | + |\mu_{tf_{k_m}}(t g^N)| + |\mu_{t}(t g^N) | \\
&\le |\mu_{tf_{k_m}}(t f^N) - \mu_{t}(t f^N) | + \frac{\epsilon}{2}\,,
\end{align*}
and recall that since $f^N$ is continuous and bounded and $\mu_t = \lim_{m \to \infty} \mu_{tf_{k_m}}$, the result follows.
\end{proof} 

Now we turn our attention to the estimates on $h(\mu_t)$. Let $\alpha = \{[a] : a \in \mathcal{A}\}$ be the natural partition of $\Sigma$, then for any $\mu \in \mathcal{M}_{\sigma}(\Sigma)$ we have to its Kolmogorov-Sinai entropy is defined by

\begin{equation}
h(\mu) = \inf_n \frac{1}{n} H(\mu \mid \alpha^n) \,.
\label{eq2ent}
\end{equation}
Since for any $m \in \mathbb{N}$ we have
\[
h(\mu_{tf_{k_m}}) = P_{k_m}(t) + \mu_{tf_{k_m}}(-tf) \,,
\]
the sequence $(\mu_{tf_{k_m}}(-tf))_{m \in \mathbb{N}}$ is bounded above and  $P_{k_m}(t) \leq P(t)$, then, there is a constant $B > 0$ such that 
\begin{equation}
h(\mu_{tf_{k_m}}) \leq P(t) + \mu_{tf_{k_m}}(-tf) \leq B \,.
\label{eq1ent}
\end{equation}
In this way, we have the following. 

 \begin{prop} For each $N \in \mathbb{N}$ there is $n_0 \geq N$ such that the sequence $(H(\mu_{tf_{k_m}} \mid \alpha^{n_0}))_{m \in \mathbb{N}}$ is bounded above.
\end{prop} 
\begin{proof} Suppose that the proposition is not true, then there exists $N_0 \in \mathbb{N}$ such that for any $n \geq N_0$ the sequence $(H(\mu_{tf_{k_m}} \mid \alpha^n))_{m \in \mathbb{N}}$ is not bounded. Let $n' \geq N_0$ and $B' > 0$, then there exists a subsequence $(k_l)_{l \in \mathbb{N}}$ of the sequence $(k_m)_{m \in \mathbb{N}}$ such that $H(\mu_{tf_{k_l}} \mid \alpha^{n'}) > B'$ for any $l \in \mathbb{N}$, particularly we have to $\frac{1}{n'}H(\mu_{tf_{k_l}} \mid \alpha^{n'}) > \frac{B'}{n'}$. Since $B'$ is choosing arbitrarily, then we can choose the sequence $(k_l)_{l \in \mathbb{N}}$ such that the last inequality is true for $B' = 2Bn'$ i.e. $\frac{1}{n'}H(\mu_{tf_{k_l}} \mid \alpha^{n'}) > 2B$ for any $l \in \mathbb{N}$. On other hand the sequence $(h(\mu_{tf_{k_m}}))_{m \in \mathbb{N}}$ is bounded above with upper bound $B$, then by \eqref{eq2ent} we obtain a contradiction. Therefore the proposition is true. 
\end{proof}

\begin{lema}\label{lement} Let $\mu_t = \lim_{m \to \infty} \mu_{tf_{k_m}}$, then for each $N \in \mathbb{N}$ there are $n_0 \geq N$ and a subsequence $(k_l)_{l \in \mathbb{N}}$ of the sequence $(k_m)_{m \in \mathbb{N}}$ such that 
\[
H(\mu_t \mid \alpha^{n_0}) = \lim_{l \to \infty} H(\mu_{tf_{k_l}} \mid \alpha^{n_0}) \,.
\]
\end{lema}
\begin{proof} Let $N \in \mathbb{N}$, by the above proposition there exists $n_0 \geq N$ such that $(H(\mu_{tf_{k_m}} \mid \alpha^{n_0}))_{m \in \mathbb{N}}$ is bounded above, this implies that there is a subsequence $(k_l)_{l \in \mathbb{N}}$ of the sequence $(k_m)_{m \in \mathbb{N}}$ such that $\lim_{l \to \infty} H(\mu_{tf_{k_l}} \mid \alpha^{n_0}) < \infty$. 

Moreover for $l_0 \in \mathbb{N}$ and any $j \geq l_0$ we have to
\[
\inf_{l \geq l_0}\{-\mu_{tf_{k_l}}[\omega]\log(\mu_{tf_{k_l}}[\omega])\} \leq -\mu_{tf_{k_j}}[\omega]\log(\mu_{tf_{k_j}}[\omega]),
\]
therefore summing up all the $[\omega] \in \alpha^{n_0}$ we obtain that
\[
\sum_{[\omega] \in \alpha^{n_0}} \inf_{l \geq l_0}\{-\mu_{tf_{k_l}}[\omega]\log(\mu_{tf_{k_l}}[\omega])\} \leq \inf_{j \geq l_0} \left\{\sum_{[\omega] \in \alpha^{n_0}} -\mu_{tf_{k_j}}[\omega]\log(\mu_{tf_{k_j}}[\omega])\right\} \,.
\]
Moreover, taking the limit as $l_0 \to \infty$ and using the monotone convergence theorem for the integrable function $\phi_m(x) = \inf_{l \geq m}\{-\mu_{tf_{k_l}}(x)\log(\mu_{tf_{k_l}}(x))\}$ with the counting measure on the set $\{[\omega] : [\omega] \in \alpha^{n_0}\}$, we get
\[
\sum_{[\omega] \in \alpha^{n_0}} \liminf_{l \to \infty}(-\mu_{tf_{k_l}}[\omega]\log(\mu_{tf_{k_l}}[\omega])) \leq \liminf_{l \to \infty} \sum_{[\omega] \in \alpha^{n_0}} -\mu_{tf_{k_l}}[\omega]\log(\mu_{tf_{k_l}}[\omega]) \,.
\]
Since $\mu_t = \lim_{l \to \infty} \mu_{tf_{k_l}}$ and $\mu_t(\partial[\omega]) = 0$ for any cylinder $[\omega] \in \alpha^{n_0}$, then we have to $\mu_t[\omega]\log(\mu_t[\omega]) = \lim_{l \to \infty} \mu_{tf_{k_l}}[\omega]\log(\mu_{tf_{k_l}}[\omega])$. 

Therefore follows that
\[
\sum_{[\omega] \in \alpha^{n_0}} -\mu_t[\omega]\log(\mu_t[\omega]) \leq \liminf_{l \to \infty} \sum_{[\omega] \in \alpha^{n_0}} -\mu_{tf_{k_l}}[\omega]\log(\mu_{tf_{k_l}}[\omega]) \,.
\]
By a similar proceeding we can prove that 
\[
\sum_{[\omega] \in \alpha^{n_0}} -\mu_t[\omega]\log(\mu_t[\omega]) \geq \limsup_{l \to \infty} \sum_{[\omega] \in \alpha^{n_0}} -\mu_{tf_{k_l}}[\omega]\log(\mu_{tf_{k_l}}[\omega]) \,,
\]
therefore we conclude that 
\[
H(\mu_t \mid \alpha^{n_0}) = \lim_{l \to \infty} H(\mu_{tf_{k_l}} \mid \alpha^{n_0}) \,.
\]
\end{proof}

\begin{prop} \label{proplimsup}Let $(\mu_{tf_{k_l}})_{l \in \mathbb{N}}$ be a sequence given by the previous lemma. Then $\limsup_{l \to \infty} h(\mu_{tf_{k_l}}) \leq h(\mu_t)$.
\end{prop}
\begin{proof}
 Suppose that $\limsup_{l \to \infty} h(\mu_{tf_{k_l}}) > h(\mu_t)$, then we can choose $\epsilon > 0$ such that $h(\mu_t) \leq \limsup_{l \to \infty} h(\mu_{tf_{k_l}}) - 3\epsilon$. From \eqref{eq2ent}, $\frac{1}{n} H(\mu_t \mid \alpha^n) \leq h(\mu_t) + \epsilon$ for $n$ big enough, and by lemma \ref{lement}, there is $n_0 \geq n$ such that $\lim_{l \to \infty} H(\mu_{tf_{k_l}} \mid \alpha^{n_0}) = H(\mu_t \mid \alpha^{n_0})$. Finally by \eqref{eq2ent}, there is $l_0 \in \mathbb{N}$ such that for $j \ge l_0$ we have $h(\mu_{tf_{k_{j}}}) \leq \frac{1}{n_0} H(\mu_{tf_{k_{j}}} \mid \alpha^{n_0})$, therefore 
\begin{align*}
h(\mu_{tf_{k_{j}}}) &\leq \frac{1}{n_0} H(\mu_t \mid \alpha^{n_0}) + \epsilon \\
								&\leq h(\mu_t) + 2\epsilon \\
								&\leq \limsup_{l \to \infty} h(\mu_{tf_{k_l}}) - \epsilon \,,
\end{align*}
and taking the $\limsup$ as $j \to \infty$ in the left side of the inequality, we obtain a contradiction. In this way, we conclude that $\limsup_{l \to \infty} h(\mu_{tf_{k_l}}) \leq h(\mu_t)$ \,.
\end{proof}

\begin{proof}{\emph{(Theorem 1)}}
Let $(k_l)_{l \in \mathbb{N}}$ be a sequence as in lemma \ref{lement}. Since $(k_l)_{l \in \mathbb{N}}$ is a subsequence of $(k_m)_{m \in \mathbb{N}}$ and the sequence $(P_{k_l}(t))_{l \in \mathbb{N}}$ is increasing, from \eqref{VPApproximation} and lemma \ref{lema_lim_ut} we have 
\[
\lim_{l \to \infty} P_{k_l}(t) = P(t)
\] 
and 
\[
\lim_{l \to \infty} \mu_{tf_{k_l}}(tf) = \mu_t(tf)\,.
\] 

Therefore 
\[
P(t) = \lim_{l \to \infty} P_{k_l}(t) \leq \limsup_{l \to \infty}h(\mu_{tf_{k_l}}) + \limsup_{l \to \infty}\mu_{tf_{k_l}}(tf_{k_l}) \leq h(\mu_t) + \mu_t(tf) \,.
\]

The above inequality shows that $\mu_t$ is an equilibrium state associated to the potential $tf$. Since $P(tf) < \infty$, $tf$ is bounded above and $V(tf) < \infty$, from \cite{BuSa03}, we have that the equilibrium state associated to the potential $tf$ is unique and, therefore, $\mu_t = \mu_{tf}$.

To prove that the family $(\mu_{tf})_{t>1}$ is tight, notice that from equation \eqref{ut_is_tight} and that $\mu_t = \mu_{tf}$, we can take the limit as $m \to \infty$ at both sides, maybe under a convergent subsequence of $\mu_{tf_k}$ and noticing that the cylinders are continuity sets, and we get
\[
\mu_{tf}([i]) \leq \exp(4V(f) + \sup f |_{[i]} - S) \,.
\]
Observe that the right side does not depend on $t$. So, following the same conclusion of lemma \ref{lematight}, we can conclude that $(\mu_{tf})_{t>1}$ is tight. Taking any subsequence $t_k \to \infty$, we find an accumulation point for $(\mu_{tf})_{t>1}$ as $t\to\infty$, and theorem 1 is proved.
\end{proof} 

A consequence of this proof is the following corollary which states that the whole sequence $(\mu_{tf_k})_{k \in \mathbb{N}}$ is in fact convergent to the equilibrium state $\mu_{tf}$. That is interesting in the sense that we can approximate the equilibrium states of the non-compact case by the ones in  the invariant compact subshifts, which is not trivial since the space is not $\sigma$-compact. 

\begin{corol} For each $t > 1$ the equilibrium states sequence $(\mu_{tf_k})_{k \in \mathbb{N}}$ converges to $\mu_{tf}$.
\end{corol}
\begin{proof} Suppose that the sequence $(\mu_{tf_k})_{k \in \mathbb{N}}$ is not convergent. Recall  $\mathcal{M}_{\sigma}(\Sigma)$ is a metrizable space and let $d$ be any distance on this space. Then, there is $\epsilon_0 > 0$ and a subsequence $(\mu_{tf_{k_j}})_{j \in \mathbb{N}}$ of the sequence $(\mu_{tf_k})_{k \in \mathbb{N}}$ such that for any $j \in \mathbb{N}$ 
\[
d(\mu_{tf_{k_j}}, \mu_{tf}) \geq  \epsilon_0 \,. 
\]

In particular, notice that $(\mu_{tf_{k_j}})_{j \in \mathbb{N}}$ cannot converge to $\mu_{tf}$.

Now, by lemma \ref{lematight}, the sequence $(\mu_{tf_{k_j}})_{j \in \mathbb{N}}$ is tight, therefore the argument used to prove theorem 1 implies that there exists $(\mu_{tf_{k_i}})_{i \in \mathbb{N}}$ be a subsequence of $(\mu_{tf_{k_j}})_{j \in \mathbb{N}}$ such that $$\lim_{i \to \infty} \mu_{tf_{k_i}} = \mu_{tf}\,.$$ Since the equilibrium state is unique, this is a contradiction, which proves the result.
\end{proof}

Observe that, by remark \ref{rem1}, the assumption that $f$ is summable is essential for the proof of the theorem \ref{teor1}, and therefore, also for theorem \ref{teor3}. Otherwise, there are counterexamples, like the one in \cite{Sar00}, where it can be checked that the potential is not summable, although it is Markov.

\section{Proof of the theorem 2}

In this section we prove the other theorem of this work. This proof is a direct consequence of the theorem \ref{teor1} and the following result. 

The following proposition shows that there exist $k_0$ large enough such that the set of the $f_k$-maximizing measures is the same for each $k \geq k_0$. This is the key point where we use the results in \cite{BiFr14} to locate the ground states.

\begin{prop}\label{propmax} There is $k_0$ such that, for each $k \geq k_0$, we have that $\beta_k = \beta$ and $\mathcal{M}_{max}(f_k) = \mathcal{M}_{max}(f)$. 
\end{prop}

\begin{proof} It is proved in \cite{BiFr14} that there exits a finite set $F \subset \mathbb{N}$ such that
\[
\beta = \sup \{ \mu(f) : \mu \in \mathcal{M}_{\sigma}(\Sigma_F)\},
\]
and every measure $\mu \in \mathcal{M}_{max}(f)$ satisfies $supp(\mu) \subset \Sigma_F$. Moreover, there exists $k_0 \geq 0$ such that $\Sigma_F \subset \Sigma_{k_0}$. Therefore, we can suppose w.l.o.g. that $\Sigma_F = \Sigma_{k_0}$. Then
\begin{align*}
\beta &= \sup \{ \mu(f) : \mu \in \mathcal{M}_{\sigma}(\Sigma_{k_0})\} \\
       &= \sup \{ \mu(f_{k_0}) : \mu \in \mathcal{M}_{\sigma}(\Sigma_{k_0})\} \\
			 &= \beta_{k_0}.
\end{align*}
Observe that for each $k$ we have $\beta_k \leq \beta_{k+1} \leq \beta$, this is because $\mathcal{M}_{\sigma}(\Sigma_k) \subset \mathcal{M}_{\sigma}(\Sigma_{k+1})$. Besides that, $\mu(f_k) = \mu(f_{k+1})$ for any $\mu \in \mathcal{M}_{\sigma}(\Sigma_k)$, then $\beta_k = \beta$ for each $k \geq k_0$ and every measure $\mu \in \mathcal{M}_{max}(f)$ satisfies $supp(\mu) \subset \Sigma_{k_0}$, moreover we have that
\[
\beta = \mu(f) = \mu(f_k)\,,
\]
that is $\mu \in \mathcal{M}_{max}(f_k)$. Let $k \geq k_0$, if the probability measure $\mu_k \in \mathcal{M}_{max}(f_k)$, using that $supp(\mu_k) \subset \Sigma_k$, we have
\[
\beta_k = \mu_k(f_k) = \mu_k(f),
\]
since $\beta_k = \beta$, then we conclude that $\mu_k \in \mathcal{M}_{max}(f)$.
\end{proof}

The following proof is basically a consequence of the theorem \ref{teor1} and the previous proposition, in addition to this we have to this result is a complete generalization of the theorem of Morris in \cite{Mor07} beyond the finitely primitive case.

\begin{proof}{\emph{(Theorem 2)}}

Since the $\Sigma_k$'s are compact, then the functions $t \mapsto h(\mu_{tf_k})$ are decreasing for each $k \in \mathbb{N}$, besides that using (9) and the Lemma 2 join to the corollary 1, follows immediately of the variational principle that $\lim_{k \to \infty}h(\mu_{tf_k}) = h(\mu_{tf})$. Therefore, for each $t_1 > t_0 > 1$ and any $k \in \mathbb{N}$ we have to $h(\mu_{t_0f_k}) > h(\mu_{t_1f_k})$, then taking the limit as $k \to \infty$ we obtain that  
\[
h(\mu_{t_0f}) = \lim_{k \to \infty}h(\mu_{t_0f_k}) \geq \lim_{k \to \infty}h(\mu_{t_1f_k}) = h(\mu_{tf})\,,
\]
i.e. the family $(h(\mu_{tf}))_{t>1}$ is non-increasing and particularly is bounded above for $t$ large enough. Using again the variational principle we obtain that
\[
\frac{h(\mu_{tf})}{t} + \mu_{tf}(f) = \sup\left\{\frac{h(\mu)}{t} + \mu(f) : \mu \in \mathcal{M}_{\sigma}(\Sigma) \text{ and } \mu(f) > -\infty \right\}\,,
\]
then taking an increasing sequence $(t_i)_{i \in \mathbb{N}}$ in $(1, \infty)$ such that $\lim_{j \to \infty} \mu_{t_jf} = \mu_{\infty}$ and using the convexity of the function $P(t)$, follows immediately that this function admits an asymptote as $t \to \infty$ with slope $\beta$, i.e. $P(t) = h + t\beta + \rho(t)$ with $h \in \mathbb{R}$ and $\lim_{t \to \infty} \rho(t) = 0$.

On other hand, for each $k \geq k_0$ we have to $\mathcal{M}_{\max}(f) = \mathcal{M}_{\max}(f_k)$, therefore $\mathcal{M}_{\max}(f)$ is a compact set of $\mathcal{M}_{\sigma}(\Sigma_{k_0})$ and so is compact in $\mathcal{M}_{\sigma}(\Sigma)$. By the upper semi-continuity of the function $\mu \mapsto h(\mu)$ restricted to the compact set $\mathcal{M}_{\max}(f)$, there exists a maximal element $\hat{\mu} \in \mathcal{M}_{\max}(f)$ i.e. a probability measure such that $h(\mu) \leq h(\hat{\mu})$ for any $\mu \in \mathcal{M}_{\max}(f)$, particularly $h(\mu_{\infty}) \leq h(\hat{\mu})$. 

Since $\hat{\mu}(f) = \beta$, by the variational principle follows that 
\begin{equation}\label{inequality 1}
h(\hat{\mu}) + t\beta(f) \leq P(t) = h + t\beta(f) + \rho(t),
\end{equation}
this inequality is valid for each $t > 1$, therefore $h(\hat{\mu}) \leq h + \rho(t)$, taking the limit as $t \to \infty$ follows that $h(\hat{\mu}) \leq h$. Besides that $\mu_{tf}$ is an equilibrium state associated to the potential $tf$, then 
\begin{equation}\label{inequality 2}
h + t\beta(f) + \rho(t) = h(\mu_{tf}) + t\mu_{tf}(f) \leq h(\mu_{tf}) + t\beta(f),
\end{equation}
i.e. $h(\mu_{\infty}) \leq \limsup_{t \to \infty} h(\mu_{tf})$.

Now we just need to prove that $\limsup_{t \to \infty}h(\mu_{tf}) \leq h(\mu_\infty)$, in fact observe that for any increasing sequence $(t_j)_{j \in \mathbb{N}}$ in $(1, \infty)$ and each $t' \geq t_j$ we have
\[
\inf_{t \geq t_j}\{-\mu_{tf}[\omega]\log(\mu_{tf}[\omega])\} \leq -\mu_{t'f}[\omega]\log(\mu_{t'f}[\omega]).
\]
Summing up all the $[\omega] \in \alpha^{n_0}$ we obtain that for each $t' \geq t_j$
\[
\sum_{[\omega] \in \alpha^{n_0}}\inf_{t \geq t_j}\{-\mu_{tf}[\omega]\log(\mu_{tf}[\omega])\} \leq \sum_{[\omega] \in \alpha^{n_0}} -\mu_{t'f}[\omega]\log(\mu_{t'f}[\omega]), 
\]
follows immediately of the above inequality that
\begin{equation}\label{inequality}
\sum_{[\omega] \in \alpha^{n_0}}\inf_{t \geq t_j}\{-\mu_{tf}[\omega]\log(\mu_{tf}[\omega])\}
\leq \inf_{t' \geq t_j} \left\{\sum_{[\omega] \in \alpha^{n_0}} -\mu_{t'f}[\omega]\log(\mu_{t'f}[\omega]\right\}.
\end{equation}
Using the monotone convergence theorem for $\phi_j(x) = \inf_{t \geq t_j}\{-\mu_{tf}(x)\log(\mu_{tf}(x))\}$ with the counting measure on the set $\{[\omega] : [\omega] \in \alpha^{n_0}\}$, follows that
\[
\lim_{j \to \infty} \sum_{[\omega] \in \alpha^{n_0}} \inf_{t \geq t_j}\{-\mu_{tf}[\omega]\log(\mu_{tf}[\omega])\} = \sum_{[\omega] \in \alpha^{n_0}} \lim_{j \to \infty} \inf_{t \geq t_j}\{-\mu_{tf}[\omega]\log(\mu_{tf}[\omega])\}.
\]
On other hand $\mu_{\infty}(\partial[\omega])=0$ for each $[\omega] \in \alpha^{n_0}$ and $\lim_{j \to \infty}\mu_{t_jf} = \mu_{\infty}$, then $\lim_{j \to \infty}\mu_{t_jf}[\omega] = \mu_{\infty}[\omega]$. Therefore taking the limit as $j \to \infty$ in both sides of (\ref{inequality}) we conclude that
\[
\sum_{[\omega] \in \alpha^{n_0}}-\mu_{\infty}[\omega]\log(\mu_{\infty}[\omega])
\leq \liminf_{t \to \infty} \sum_{[\omega] \in \alpha^{n_0}} -\mu_{tf}[\omega]\log(\mu_{tf}[\omega]).
\]
An analogous proof shows that $\limsup_{t \to \infty}H(\mu_{tf}|\alpha^{n_0}) \leq H(\mu_{\infty}|\alpha^{n_0})$, therefore $\lim_{t \to \infty}H(\mu_{tf}|\alpha^{n_0}) = H(\mu_{\infty}|\alpha^{n_0})$. 

Then following the same proof of proposition 2 we obtain that 
\[
\limsup_{t \to \infty}h(\mu_{tf}) \leq h(\mu_{\infty}).
\]
Finally using the inequalities (\ref{inequality 1}) and (\ref{inequality 2}), follows that 
\[ 
h(\hat{\mu}) \geq h(\mu_{\infty}) \geq \limsup_{t \to \infty}h(\mu_{tf}) \geq h \geq h(\hat{\mu}),
\]
and this concludes our proof.
\end{proof}

\section*{Acknowledgments}
The authors are extremely grateful to professor Rodrigo Bissacot for reading the first drafts, talks and suggestions that greatly improved the final exposition in this paper. We would also like to thank Jose Chauta, who identified a gap in an earlier version of the paper, and Prof. Mariusz Urba\'nski for the nice report that helped us finish this work. 

\bibliographystyle{abbrv}

\end{document}